\documentclass[10pt,twocolumn]{IEEEtran}      
\usepackage{latexsym}
\usepackage{amssymb}
\usepackage{amsbsy}
\usepackage{textcomp}
\usepackage{amsmath}
\usepackage{color}
\usepackage{cite}
\usepackage{epsfig}
\usepackage[amssymb]{SIunits}
\usepackage{graphicx}
\usepackage{psfig}

\newcommand{\ncom}{\newcommand}
\ncom{\beqn}{\begin{eqnarray*}} \ncom{\eeqn}{\end{eqnarray*}}
\ncom{\beq}{\begin{eqnarray}} \ncom{\eeq}{\end{eqnarray}}
\newtheorem{thm}{Theorem}[section]
\newtheorem{assum}{Assumption}[section]
\newtheorem{pro}{Proposition}[section]
\newtheorem{cor}[thm]{Corollary}
\newtheorem{definition}{Definition}[section]
\newtheorem{remark}{Remark}[section]
\newtheorem{lemma}{Lemma}[section]
\newtheorem{example}{Example}[section]
\ncom{\ben}{\begin{enumerate}} \ncom{\een}{\end{enumerate}}
\ncom{\bl}{\begin{lemma}} \ncom{\el}{\end{lemma}}
\ncom{\bt}{\begin{thm}} \ncom{\bn}{\begin{notation}}
\ncom{\en}{\end{notation}} \ncom{\et}{\end{thm}}
\ncom{\bp}{\begin{pro}} \ncom{\ep}{\end{pro}}
\ncom{\brm}{\begin{remark}} \ncom{\erm}{\end{remark}}
\ncom{\bex}{\begin{example}} \ncom{\eex}{\end{example}}
\ncom{\bd}{\begin{definition}} \ncom{\ed}{\end{definition}}
\ncom{\bc}{\begin{cor}} \ncom{\ec}{\end{cor}}
\ncom{\ba}{\begin{assum}} \ncom{\ea}{\end{assum}} \ncom{\pf}{{\bf
Proof: }} \ncom{\R}{I\!\!R}

\def\deff{\stackrel{\triangle}{=}}

\ncom{\doq}{\nabla _q} \ncom{\dop}{\nabla _p} \ncom{\dopt}{\nabla
_{\tilde{p}}}
\newcommand{\vectornorm}[1]{\vert\vert #1\vert \vert}
\def\norm#1{\|#1\|}
\DeclareMathOperator{\rank}{rank} \DeclareMathOperator{\diag}{diag}
\DeclareMathOperator{\Col}{Col} \DeclareMathOperator{\Null}{Null}
 
 \DeclareMathOperator{\Span}{Span}
 \def\sign{\mbox{\em sign}}

\begin{document}
\title{Synchronization and semistability analysis   of the  Kuramoto model of coupled nonlinear oscillators}
\author{Vishaal Krishnan\thanks{Vishaal is a graduate student in the Department of Mechanical Engineering, Indian Institute of Technology
Madras, Chennai-600036, India, (email:
vishaalkrishnan1992@gmail.com).}, Arun D.
Mahindrakar,~\IEEEmembership{Member,~IEEE}\thanks{Arun D.
Mahindrakar is with the Department of Electrical Engineering, Indian
Institute of Technology Madras, Chennai-600036, India, (email:
arun\_dm@iitm.ac.in).} and Somashekhar S. Hiremath
\thanks{Somashekhar S. Hiremath is with the Department of Mechanical Engineering, Indian
Institute of Technology Madras, Chennai-600036, India, (email:
somashekhar@iitm.ac.in).}}
\maketitle
\begin{abstract}
An interesting problem in synchronization is the study of coupled
oscillators, wherein oscillators with different natural frequencies
synchronize to a common frequency and equilibrium phase difference.
In this paper, we investigate the stability and convergence
in a network of coupled oscillators described by the Kuramoto model.
We consider networks with finite number of oscillators, arbitrary interconnection topology, non-uniform coupling gains and non-identical natural frequencies. We show that such a network synchronizes provided the underlying graph is connected and certain conditions on the coupling gains are satisfied. In the analysis, we consider as states the phase and angular frequency differences between the oscillators, and the resulting dynamics possesses a continuum of equilibria. The synchronization problem involves establishing the Lyapunov stability of the fixed points and showing convergence of trajectories to these points. The synchronization result is established in the framework of semistability theory.

\end{abstract}
\section{Introduction}
 Synchronization as a control problem in networked
multi-agent systems has always assumed great importance. In this
context, the study of synchronization in natural and engineered
systems is of significant value, as it would serve to further the
understanding of the phenomenon. The Kuramoto model, in which
oscillators in the network spontaneously synchronize for coupling
gains above a certain value, is of particular interest in this
regard. The Kuramoto model has been used in the past to study a
variety of systems, biological systems such as networks of pacemaker
cells in the heart, circadian pacemaker cells in the brain, laser
arrays and superconducting Josephson junctions \cite{Strogatz}.

The study of synchronization in networked dynamical systems entails
an analysis of stability of synchronized states and convergence of
trajectories to these states. It is also important to obtain
necessary and sufficient conditions for synchronization and
characterize the basins of attraction corresponding to the
synchronized states.

The interest in synchronization problem can be gauged by the vast
body of related  literature in the area of networked  dynamical
systems. In the paper \cite{MirolloStrogatz}, the authors discuss
synchronization in pulse-coupled biological oscillators. The article
\cite{Strogatz} reviews work on the Kuramoto model. Although
published years ago, it serves to be a very useful starting point.
Subsequently, in \cite{Jadbabaie}, the authors prove local
exponential
 convergence for the Kuramoto model of coupled oscillators, individual oscillators having
  identical natural frequencies with uncertainties, from a control-theoretic
  viewpoint, while in \cite{spong05}, the authors derive a lower bound on the
   coupling gain for the onset of synchronization and a lower bound on the same
    gain which is sufficient for synchronization, with exponential convergence. The authors
     in \cite{Acebron} present a review of the work on synchronization in the
      Kuramoto model. In \cite{Dorfler},  the necessary
and sufficient conditions on the critical coupling to achieve
synchronization in the Kuramoto model are derived and in \cite{bullo_Dorfler} the authors
analyze the non-uniform Kuramoto model and its equivalence to the classical swing equation in power networks.
In \cite{Mirollo}, the authors analyze the linear stability of the
phase-locked state in the Kuramoto model.  Robustness of the
phase-locking in the Kuramoto model subjected to time-varying
natural frequencies is analyzed  in \cite{Antoine}. In a recent
development the authors in \cite{Wang} have shown the exponential
synchronization convergence of Kuramoto oscillator in the presence
of a pacemaker. The pacemaker, also called as pinner, supplies
reference phase to the rest of the network and the objective is to
synchronize the phases of the rest of the network to the reference
phase. In \cite{Maistrenko}, the authors illustrate the mechanism of
desynchronization with decreasing coupling strength in a
$3$-oscillator network described by the Kuramoto model, and present
numerical illustrations for the $N=5$ case. In \cite{Franci}, the
authors use proportional mean-field feedback control to achieve
desynchronization in the Kuramoto model for oscillators with small
natural frequencies. Analysis of limit cycles in interconnected
oscillators, using a method based on energy exchange, considering
the oscillator networks as open systems is presented in
\cite{Guy_Bart_Stan}. As for relevant work in synchronization of
multi-agent systems, the authors in \cite{Scardovi} study
synchronization in networks of identical linear systems. In a recent
paper \cite{Trentelman}, the authors derive protocols for robust
synchronization in uncertain linear multi-agent systems.

 Lyapunov stability of an equilibrium point
guarantees that for initial conditions close to the equilibrium
point, the trajectories remain arbitrarily close
 to it, while asymptotic stability guarantees the convergence of trajectories to
 the equilibrium point, for initial conditions in its neighbourhood. Naturally, asymptotic
 stability implies Lyapunov stability. The two key ideas are those of
 stability (in the sense of Lyapunov) and convergence of solutions. The concept
 of semistability finds relevance in systems that possess a continuum
 of (or non-isolated) equilibria. In any neighbourhood of a non-isolated equilibrium point there
 exists another equilibrium point, which implies that such an equilibrium point cannot be
 asymptotically stable. However, for initial conditions in the neighbourhood of such an
 equilibrium point, it is still of interest to study the stability and
 convergence of solutions (to possibly different equilibrium points). The concept of
 semistability of an equilibrium point essentially encompasses these two notions, of
 stability (in the sense of Lyapunov) of the equilibrium point, and convergence of solutions.

We analyze the stability in the Kuramoto model from a control
theoretic viewpoint. For the problem of synchronization in the
Kuramoto model, we consider the dynamics of phase and angular
frequency differences between oscillators in the network, by a
linear transformation from the space of phase angles and angular
frequencies of oscillators in the network. The natural frequencies
of the individual oscillators in such a setting influence the
initial conditions of the system, which possesses a continuum of
equilibria. The set of all initial conditions in the Kuramoto model
corresponds to an equivalent subset of initial conditions of the
system. The problem of synchronization of oscillators in the network
now reduces to that of convergence of trajectories to limit points
in the equilibrium set, and the stability of the limit points.
Semistability, as a notion of stability is appropriate in this
context. We consider the general case of the Kuramoto model with non-identical natural
frequencies, non-uniform (but symmetric) coupling with the underlying graph corresponding to the network being connected.
We  provide conditions on the coupling strengths and establish that for all initial phases
in a compact set contained in a half-circle, the oscillators synchronize. We further show that the set $[-\pi/2,\pi/2]$
is an attracting set in $(-\pi,\pi]$ for the phase differences, and through semistability analysis establish that the
trajectories converge to stable fixed points, which form the contribution of this paper.

The outline of the paper is as follows.  In subsection
\ref{background} we introduce the necessary background material on
semistability of continuous systems. In section \ref{model} we
introduce the Kuramoto model, define various notions of
synchronization, and follow it up with the formulation of the
problem for the case of two coupled oscillators for illustrative
purposes. In section \ref{2semi} we derive semistability and global
synchronization results for the two-oscillator case. Section
\ref{Nmodel} contains the analysis for the general case of networks
with finite $N$ coupled oscillators, with non-uniform coupling,
arbitrary natural frequencies and arbitrary interconnection
topology. In section \ref{num}, we present simulation results and conclude with a summary of results
and contributions in Section \ref{con}.
\subsection{Notations and preliminaries}
Let $\norm{\cdot} $ denote the  2-norm  on $\R^n$. The unit circle on
$\R^2$ is given by ${S}^{1} = \{x \in \R^2 : \norm{x} = 1\}$ and we
denote by $\Gamma^n\deff \underbrace{S^1\times S^1\times \cdots
\times S^1}_{n}$, the $n$-torus.
 Let $\overline{\mathcal{K}}$ denote the
closure of $\mathcal{K}$, where $\mathcal{K} \subseteq \R^n$. Two subsets
$A$ and $B$ of $\R^n$ are said to be separated if both
$\overline{A}\cap B$ and $A\cap \overline{B}$ are empty. A set
$\mathcal{K} \subseteq \R^n$  is said to be connected if $\mathcal{K}$ is
not a union of two non-empty separated sets. The connected component
of $\mathcal{K}$ is the maximal connected subset of $\mathcal{K}$.
A set $\mathcal{K}$ is called convex if $\alpha x+ (1-\alpha)y \in
\mathcal{K}$ for all \ $x,y \in \mathcal{K}$\ and $ \alpha \in
[0,1]$. The convex hull of $\mathcal{K}$, denoted by conv
$\mathcal{K}$, is the intersection of all convex sets containing
$\mathcal{K}$ and co $\mathcal{K}$ denote the union of convex hulls
of the connected components of $\mathcal{K}$. The cone generated by co
$\mathcal{K}$ is denoted by coco $\mathcal{K}$. Let $\Col(A)$ denote
the column space of a matrix $A$, while $\Null(A)$ denotes its
nullspace. For
$X=(x_1,x_2,\ldots,x_n)\in \Gamma^n$, $\sin(X)\deff
(\sin(x_1),\sin(x_2),\ldots,\sin(x_n))$. We denote by  ${\pmb 1}_N $ an $N$-dimensional vector with all components equal to one.
%
\subsection{Background on semistability of continuous systems}
\label{background}
 Consider the dynamical system described by
\begin{equation}
\dot{x}(t)  =  f(x(t)),
\label{eq:sde}
\end{equation}
where $f:D\longrightarrow \ \R^n$ is continuous on an open and connected set $D\subseteq \ \R^n$.
By Peano's existence theorem \cite{hartman}, there exists a $t_1>0$ such that on $[0, t_1]$, the differential equation \eqref{eq:sde} possesses a
continuous solution $x:[0, t_1] \longrightarrow {D}$. Further we assume that the solution is $C^1$ and unique.
Let $\psi(t,x_0)$ denote the solution of \eqref{eq:sde} that exists for all $t \in [0,\infty)$
and satisfies the initial condition $x(0)=x_0$.
These assumptions imply that the map $\psi:[0,\infty)\times D\longrightarrow D$ is continuous, satisfies $\psi(0,x_0)=x_0$ and possesses
the semi-group property $\psi(t_1,\psi(t_2,x))=\psi(t_1+t_2,x)$ for all $t_1,t_2\ge 0$
and $x\in D$. Given $t\in \R$, we denote the map $\psi(t,\cdot):D\longrightarrow D$ by $\psi_t$.
A point $x_e\in D$ satisfying $f(x_e)=0$ or $\psi(t,x_e) = x_e$ for
all $t \geq 0$ is an equilibrium point of \eqref{eq:sde}. The
collection of all equilibrium points of \eqref{eq:sde} is the set of
equilibria, denoted by $\mathcal{E}$. The system is said to possess a continuum of
equilibria if the set $\mathcal{E}$ has no isolated points. 
We list only the key  definitions from \cite{bhat_nontangency} related to semistability analysis
and for details, we refer the reader to the same reference.
\begin{definition}
The system \eqref{eq:sde} is convergent if, for
every $x \in \R^n$, $\lim_{t \rightarrow \infty} \psi(t,x)$ exists,
is a singleton.
\end{definition}
\begin{definition}
An equilibrium point $x \in \R^n$ is semistable
if there exists a open
neighbourhood $\mathcal{U} \subseteq \R^n$ of $x$ such that, for
every $z \in \mathcal{U}$, $\lim_{t \rightarrow \infty} \psi(t,z)$
exists, and is Lyapunov stable.
\end{definition}
\begin{definition}
Given $x \in \mathcal{G}, \; \mathcal{G} \subseteq \R^n$, the direction cone $\mathcal{F}_x$ of $f$ at $x$ relative to $\mathcal{G}$ is the intersection of all sets of the form  $\overline{\mathrm{coco}\ (f(\mathcal{U}) \setminus  \{0 \})}$, where $\mathcal{U} \subseteq \mathcal{G}$ is a relatively open neighbourhood of $x$.
\label{def:direction_cone}
\end{definition}
If $\mathcal{K}$ is a smooth submanifold of $\mathbb{R}^n$, then $T_{x}\mathcal{K}$ is the usual tangent
space \cite{isidori} to $\mathcal{K}$ at $x$.
The vector field $f$ is nontangent to the set $\mathcal{K}$ at the point
$x \in \mathcal{K}$ if $T_{x}\mathcal{K} \cap
\mathcal{F}_x \subseteq \{0\}$.
The following results from \cite{bhat_nontangency} establishes
semistability result based on the sufficient condition of
nontangency.
\bc \label{cor:7.2} Let ${\cal G} \subseteq \R^n$ be positively
invariant. Suppose $V :{\cal G} \longrightarrow \R$ is a continuous
function such that $\dot{V}$ is defined on ${\cal G}$. Let $x \in
\dot{V}^{-1}(0)$ be a local maximizer of $\dot{V}$ relative to
${\cal G}$ and a local minimizer of $V$ relative to the set ${\cal
K} \triangleq {\cal G}\setminus \overline{\dot{V}^{-1}(0)}$. Then
\begin{itemize}
\item [(i)] If $f$ is nontangent to $V^{-1}(V(x))$ at $x$ relative to ${\cal G}$, then $x$ is a Lyapunov stable equilibrium relative to ${\cal G}$.
\item [(ii)] If $f$ is nontangent to $\overline{\dot{V}^{-1}(0)}$ at $x$ relative to ${\cal G}$, then $x$ is a Lyapunov stable equilibrium relative to ${\cal G}$.
\item [(iii)] If there exists a relatively open neighbourhood ${\cal U} \subseteq {\cal G}$ of $x$ such that every equilibrium in $U$ is a local minimizer of $V$ relative to ${\cal K}$ and  $f$ is nontangent to $\overline{\dot{V}^{-1}(0)}$ at every point in ${\cal U} \cap \overline{\dot{V}^{-1}(0)}$ relative to ${\cal G}$, then $x$ is a semistable equilibrium relative to ${\cal G}$.
\item [(iv)] If $x$ is an isolated point of the set ${\dot{V}^{-1}(0)}$, then $x$ is an asymptotically stable equilibrium relative to ${\cal G}$.
\end{itemize}
\ec
\section{Kuramoto model}
\label{model}
 The Kuramoto model \cite{spong05} for a  group of $N$ oscillators with symmetric  coupling  between them is governed by
\beq
    \dot{\theta}_i = \omega_i + \sum_{j=1, j\neq i}^{N} \frac{K_{ij}}{N} \sin(\theta_j - \theta_i )
\label{sys}
\eeq
where $\theta_i\in S^1$ is the phase of the $i$-th oscillator,
$\omega_i$ is its natural frequency and $ K_{ij}>0$ is the coupling
gain. Both $\omega_i$ and $K_{ij}$ are assumed to be constants.

We need   the notion of phase-locking/exact synchronization
\cite{Antoine} associated with \eqref{sys}. A solution
$\theta^{\ast}$ of \eqref{sys} is said to {\it  synchronize} if and
only if \beqn \dot{\theta}_i^{\ast}-\dot{\theta}_j^{\ast}\rightarrow
0 \;\;\mbox{as} \;\; t\rightarrow  \infty \;\forall \;
i,j,=1,\ldots, N \eeqn
 and, further it is said to {\it  exactly synchronize} if
 \beqn
{\theta}_i^{\ast}-{\theta}_j^{\ast}\rightarrow 0 \;\; \mbox{as} \;
\;t \rightarrow \infty \;\; \forall \;i,j,=1,\ldots, N. \eeqn
Further, the network \eqref{sys} is said to {\it  globally
synchronize} if \beqn
\dot{\theta}_i^{\ast}-\dot{\theta}_j^{\ast}\rightarrow 0
\;\;\mbox{as} \;\; t\rightarrow  \infty \;\forall \; i,j,=1,\ldots,
N \eeqn for all initial conditions $\theta_i(0), i=1,\ldots,N$.
The dynamics of the  Kuramoto network with two  oscillators is first
considered.

By letting $\Delta\omega=\omega_1-\omega_2, K_{12}=K$ and $\Delta
\theta=\theta_1-\theta_2$, the dynamics of the phase difference
between the two oscillators is
\beq
\begin{array}{lcl}
\Delta\dot{\theta} &= &\Delta\omega - K \sin(\Delta\theta)\\
\Delta\ddot{\theta}& = &-K\Delta\dot{\theta} \cos(\Delta\theta),
\label{pdyn}
\end{array}
\eeq
Defining  $x_1=\Delta\theta$  and $x_2=\Delta\dot{\theta}$, the  state-space representation of the system  is given by
\beq
\dot{x}=f(x)=\left(
\begin{array}{c}
 x_2\\
-K x_2 \cos{x_1} \label{odyn}
\end{array}\right)
\eeq
and the associated  set of equilibria of \eqref{odyn} is
$\mathcal{E} = \{ (x_1,x_2) :  x_2= 0\}$. A phase-locking solution
corresponds to  an equilibrium solution $x^{\ast}\in \mathcal{E} $
of the  dynamics \eqref{odyn}.
%
In the following section we present semistability
analysis for the two oscillator case.
\section{Semistability analysis of the two oscillator Kuramoto model}
\label{2semi}
 In \cite{bhat_nontangency}, the authors  derived
nontangency-based Lyapunov function results to show convergence and
semistability for continuous systems. These results do not require
the sign definiteness of the Lyapunov function. Instead, they need
the derivative of the Lyapunov function to be nonpositive and the
equilibrium to be a local minimizer of the Lyapunov function on the
set of points where the Lyapunov function derivative is negative.
The weaker assumptions on Lyapunov function in showing semistability
are supplemented by considering nontangency of the vector field to
invariant or negatively invariant subsets of the zero-level subset
of the Lyapunov function derivative. We apply the nontangency based
semistability results to the Kuromoto model.

Define $\mathcal{E}_1 = (-\pi/2,\pi/2)\subset \mathcal{E} $; the connected component of the equilibria that is stable. We  characterize a set $\mathcal{G}$ that contains $\mathcal{E}_1$ and is positively invariant. Let $x_2=h(x_1),x_1\in(-\pi/2,\pi/2)$, where $h$ is a smooth function,  define the upper and lower boundary of the set $\mathcal{G}$. The the upper and lower boundaries correspond to the trajectories of the system \eqref{odyn} with appropriate initial conditions. On differentiating $x_2=h(x_1)$ along the trajectories of \eqref{odyn}, we obtain
 \beqn
 \dot{x}_2&=& \frac{\partial h}{\partial x_1}x_2
 \eeqn
 which on integrating,
  \beqn
  h(x_1)=-K\sin{x_1}+C
  \eeqn
  where, $C=\Delta \omega$ is the constant of integration. Using the limiting boundary conditions  $(x_1,x_2)=(\pi/2,0)$ for  the upper boundary and    $(x_1,x_2)=(-\pi/2,0)$ for the lower boundary, $\mathcal{G}$ is defined as
  \beqn
  \mathcal{G} &=& \{(x_1,x_2) : x_1\in(-\pi/2,\pi/2),\\&&x_2\in [-K(1+\sin{x_1}), K(1-\sin{x_1})]\}.
  \eeqn
 The  vector field plot of the system along with the set $\mathcal{G}$ is shown  in  Figure \ref{phase}, where
 trajectories are trapped in  ${\mathcal G}$. Through the following Lemma we show that $\mathcal{G}$  is  positively invariant.
\bl The  set ${\mathcal G}$ is  positively invariant along the
trajectories of \eqref{odyn}.
\el
\begin{proof}
 The boundary of $\mathcal{G}$ can be expressed as $\partial {\cal G}  = B_{x_1} \cup B_{x_2}$.
\beqn
B_{x_1} = \left\{
\begin{array}{lcl}
B_{x_1}^{+}& =& \{(x_1,x_2): x_1 = \frac{\pi}{2} \} \\
B_{x_1}^{-} &= &\{(x_1,x_2): x_1 = -\frac{\pi}{2} \}
\end{array}\right.
\eeqn

\beqn B_{x_2} = \{x: x_1 \in \left[ -\frac{\pi}{2}, \frac{\pi}{2}
\right], x_2 = C - K \sin(x_1),C = \pm K \} \eeqn

The normal to the boundary $B_{x_1}^{+}$ is $e_1$. The condition
\beqn e_1^\top \left[ \begin{matrix} x_2 \\
-Kx_2\cos(x_1) \end{matrix} \right] \leq 0 \eeqn
ensures that the trajectories at  $B_{x_1}^{+}$ point into  $\mathcal{G}$.
 We
have
\beqn e_1^\top \left[ \begin{matrix} x_2 \\ -Kx_2\cos(x_1)
\end{matrix} \right] = x_2. \eeqn
From the definition of $\mathcal{G}$,
$-K-K\sin(x_1) \leq x_2 \leq K-K\sin(x_1)$. Therefore, at the boundary of
$B_{x_1}^{+}$,  $x_1 = \frac{\pi}{2}$, $-2K \leq
x_2 \leq 0$. The same can be shown for the boundary $B_{x_1}^{-}$.

The boundary $B_{x_2}$ is a level set of the form $\eta(x_1,x_2) = x_2 + K\sin(x_1) = \pm K$.
The normal to the boundary $B_{x_2}$ is $\nabla_{x} \eta = [\begin{matrix} K\cos(x_1) & 1 \end{matrix}]$.
 The dot product of the normal with the vector field $f$ is  $\nabla_x\eta \left[ \begin{matrix} x_2 \\ -K\cos(x_1)\end{matrix} \right] = 0$.
 The vector field $f$ is tangential to the boundary $B_{x_2}$. Hence the claim.
\end{proof}
  Stability is thus ensured when the phase difference  lie in an open
half-circle and  difference in angular
 frequencies satisfy $\vert \Delta\omega\vert \le K$.
 \brm
 Note that exact synchronization is ensured only when the oscillators have identical natural
 frequencies, that is,  $\Delta \omega=0$.
 \erm
 \begin{figure}
\centering {\includegraphics[scale=0.5]{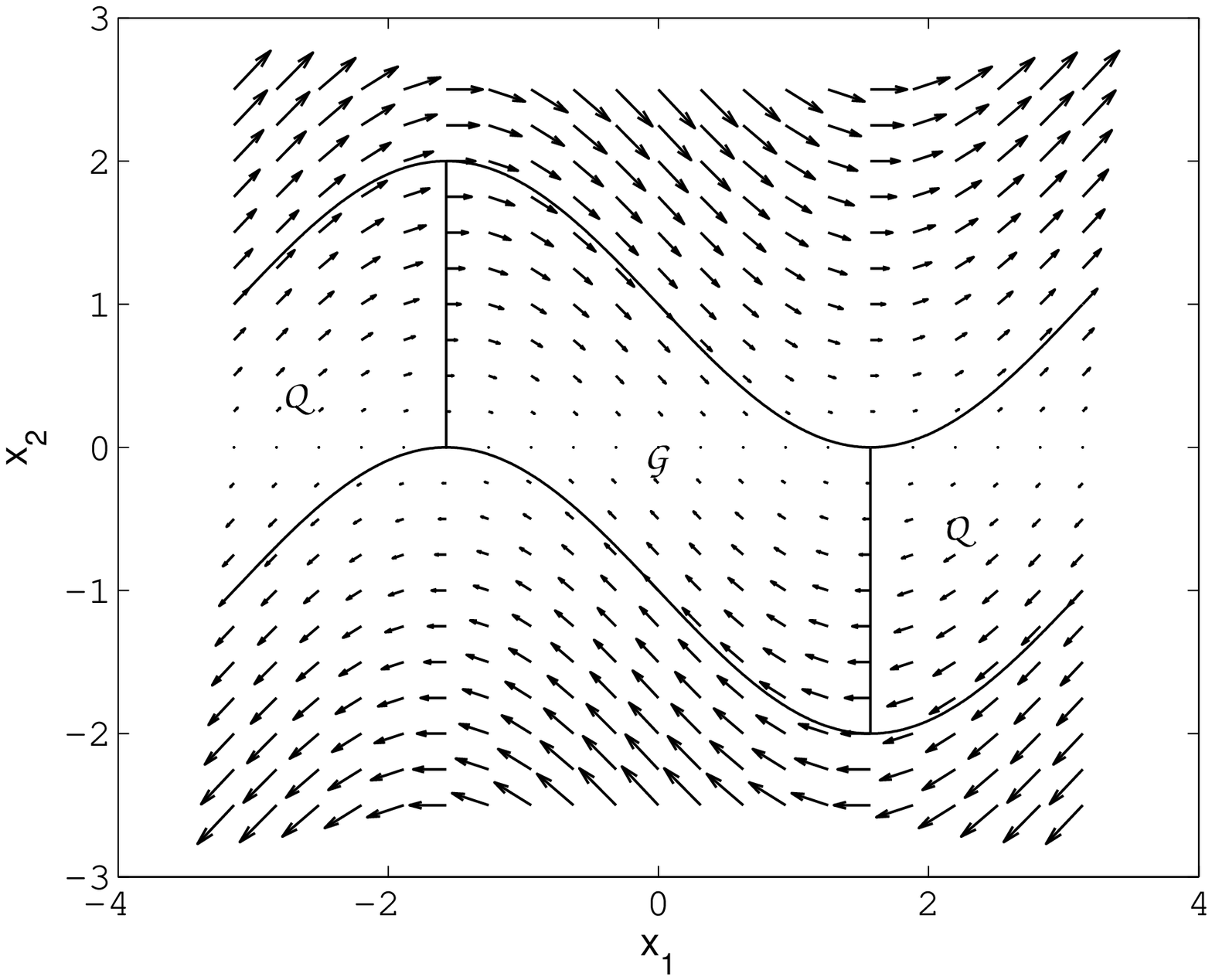}} \caption{Vector
field plot of the two-oscillator system  with $K=1$} \label{phase}
\end{figure}
We next show that every equilibrium in
$\mathcal{E}_1$ is semistable.
\subsection{Direction cone, tangent cone and nontangency}
A primary step in establishing the semistability result following
the approach given in \cite{bhat_nontangency} is to verify  the
sufficient condition of nontangency. This further requires the
computation of the direction cone and the tangent cone. In hitherto
published results \cite{bhat_nontangency, bhat_nonsmooth} the
direction cone is obtained by computing the limiting direction set
or  by expressing the vector field as a span of a set of linearly
independent vectors. If the aforementioned approaches fail in
characterizing the direction cone, then the direction cone is
computed by invoking the definition. The application of the
definition \ref{def:direction_cone} requires an  explicit
computation of the image set $f(\mathcal{U}) \setminus \{0\}$, which
is tedious and can be circumvented
 by an outer bounding set. Further, by avoiding the intersection over
 all open sets $\mathcal{U}$ in  $\overline{\mathrm{coco}\ (f(\mathcal{U}) \setminus  \{0 \})}$, leads to a
 superset of the direction cone. The resulting set, we call it  as an \textit{outer estimate} of the direction
 cone, denoted by $\hat{\mathcal{F}}_x$. The construction of
 this estimate is made clear by working out the direction  cone through the Kuramoto model.

Let $x= [a ~~ 0]^{\top} \in \mathcal{E}$, where $a\in \R$  and $\mathcal{U}\subset {\cal G}$ be a open
and bounded neighbourhood of $x$. We first consider the case $a\in (0,\pi/2)$.
Define $\mathcal{U} = \{z~\in~ \R^2:~ \parallel ~z-x \parallel_{\infty} < \epsilon \}$, where $\epsilon > 0$ is
such that $a+\epsilon <\pi/2$.
Then $\mathcal{U}$ consists of the following connected components:
%
\beqn
\begin{array}{ccl}
 A_1 & = & \{(x_1, x_2)  : a- \epsilon < x_1 < a+ \epsilon, 0 < x_2 < \epsilon \} \\
 A_2 & = & \{(x_1, x_2)  : a- \epsilon < x_1 < a+ \epsilon, -\epsilon < x_2 < 0 \} \\
 A_3 & = & \{(x_1, x_2)  : a- \epsilon < x_1 < a+ \epsilon, x_2 = 0 \}.
\end{array}
\eeqn
Now, $f(\mathcal{U})\setminus \{0\}= f(A_1) \cup f(A_2) \subset
\mathcal{A}^+ \cup \mathcal{A}^-$ where,
\beqn
\begin{array}{ccl}
\mathcal{A}^- & = & \left\{(y_1,y_2) \in \R^2: 0 < y_1 < \epsilon,\right.\\ &&\left.-K\cos(a-\epsilon)  < \frac{y_2}{y_1} <-K\cos(a+\epsilon),\right. \\
&&\left.-K\epsilon\cos(a-\epsilon) < y_2 <0 \right\}, \\
\mathcal{A}^+ & = & \left\{(y_1,y_2) \in \R^2:-\epsilon < y_1 < 0,\right.\\
 &&\left.- K\cos(a-\epsilon)< \frac{y_2}{y_1} <-K\cos(a+\epsilon),\right.\\
 &&\left. 0 < y_2 < K\epsilon \cos(a-\epsilon) \right\}.
\end{array}
\eeqn
By performing similar analysis for the case $a\in (-\pi/2,0)$, for every $x \in \mathcal{E}_1$, and for every $a\in(-\pi/2,\pi/2)$
the outer estimate of the direction cone (see Figure \ref{cones})  is
 \begin{figure}
\centering
{\includegraphics[scale=0.5]{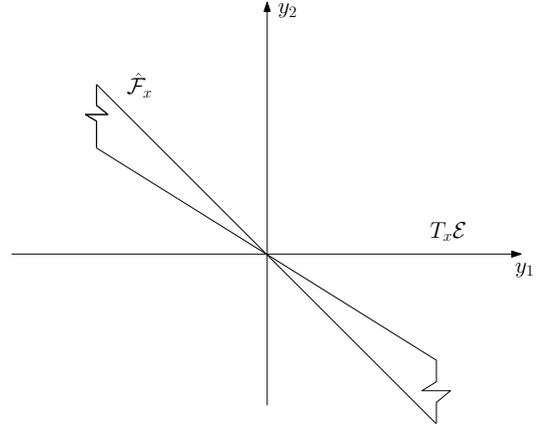}}
\caption{Tangent cone and outer estimate of direction cone}
\label{cones}
\end{figure}
\beqn
\hat{\mathcal{F}}_x &=& \left\{(y_1,y_2) \in \R^2: y_2 = p y_1,\right. \\
&&\left. p \in \left[-K\cos(a+\epsilon),
-K\cos(a-\epsilon)\right]\right\}. \eeqn
For every $x \in \mathcal{E}$, the tangent cone  is $T_x\mathcal{E}
=\{(c,0):c \in \R\}$. It now follows that for every $x \in
\mathcal{E}_1$, the intersection of tangent cone with the outer
estimate of the direction cone is $\{0\}$. The same fact is captured
in Figure \ref{cones}. Hence,  the nontangency condition holds for
every $x\in\mathcal{E}_1 $. The following result establishes the
semistability of \eqref{odyn}.
\bp
Every equilibrium in $\mathcal{E}_1$ of \eqref{odyn} is semistable relative to $\mathcal{G}$.
\label{semistability_two}
\ep
\begin{proof}
Consider the continuously differentiable function $V_1:\mathcal{G}
\longrightarrow \R$ defined by $V_1(x)=\frac{x_2^2}{2}$.
The derivative of $V_1$ along the trajectories of \eqref{odyn} is
$\dot{V}_1(x)  =  -  x_2^2\cos{x_1}\ \leq 0$  for every $x \in
\mathcal{G}$.
The set of points where the derivative of $V_1$ is zero is given by
$\dot{V}_1^{-1}(0) = \{(x_1, x_2) \in \mathcal{G} : x_2 = 0 \}$. The
largest negatively invariant subset of
$\overline{\dot{V}_1^{-1}(0)}$ is $ \mathcal{E}_1$. Let $\mathcal{K}
\triangleq \mathcal{G} \setminus \overline{\dot{V}_1^{-1}(0)}$.
Clearly, every equilibrium $z \in \mathcal{E}_1$ is a local
maximizer of $\dot{V}_1$ relative to $\mathcal{G}$ and a local
minimizer of $V_1$ relative to $\mathcal{K}$.

Consider an equilibrium $x \in \mathcal{E}_1$. There exists a
relatively open neighbourhood $\mathcal{V} \subseteq \mathcal{G}$ of
$x$ such that $\mathcal{V} \cap \mathcal{E} = \mathcal{V} \cap
\mathcal{E}_1$. It now follows that every $z \in \mathcal{V} \cap
\mathcal{E}$ is a local maximizer of $\dot{V}_1$ relative to
$\mathcal{G}$ and a local minimizer of $V_1$ relative to
$\mathcal{K}$. Moreover, ${f}$ is nontangent to $\mathcal{V} \cap
\mathcal{E}$ at every point $z \in \mathcal{V} \cap \mathcal{E}$
relative to $\mathcal{G}$.
Now, by applying (iii) of Corollary \ref{cor:7.2}, every $x \in
\mathcal{E}_1$ is semistable relative to $\mathcal{G}$.
\end{proof}

We end this section with the proof for the global synchronization in the two-oscillator network.
\subsection{Global synchronization}
\label{2global} We first note that when $| \Delta \omega | \leq K$,
all initial conditions of \eqref{sys} belong to the non-empty set
${\cal R}
\deff  {\cal Q} \cup {\cal G} $, where ${\cal Q}$ is defined by
\beqn
    {\cal Q} = \left\{x: x_1 \in (-\pi,\pi]\setminus [-\pi/2,\pi/2],\right.\\
\left. x_2 = \Delta \omega - K \sin(x_1), |\Delta \omega | \leq K \right\}.
\eeqn
The global synchronization result is established through the following proposition.
\bp
The oscillators in \eqref{sys}  synchronize for every initial condition $\theta_1(0), \theta_2(0)$, when  $| \Delta \omega | \leq K$.
\ep
\begin{proof}
We first recall from Proposition \ref{semistability_two} that every
equilibrium in ${\cal E}_1$ is semistable relative to ${\cal G}$ and
that the vector field $f$ at every point in $B_{x_1}\setminus {\cal
E}$  points into ${\cal G}$.  It suffices to show that there exists
$T\ge 0$ such that
 $\psi_T({\cal R})\cap ({\cal G} )\ne \emptyset $, where $\psi(t,x_0), t\ge 0$ denotes the solution to \eqref{odyn} corresponding to the initial condition $x(0)=x_0$.  The set $\psi_T({\cal R})\cap ({\cal G} )$ is empty if and only if there exist stable fixed points and/or closed orbits in ${\cal Q}$. Suppose {\it ad absurdum} the set $\psi_T({\cal R})\cap ({\cal G} )$ is empty.

It can be easily verified that the fixed points in ${\cal Q}$, given by ${\cal E} \cap {\cal Q} = \left\{  (x_1,x_2): x_1 \in (-\pi,\pi]/ \left[-\pi/2,\pi/2 \right]; x_2 = 0 \right\} $ are unstable.

Moreover,
\beq
    \frac{\partial f_1}{\partial x_1} + \frac{\partial f_2}{\partial x_2} =  - K \cos(x_1) > 0
\eeq in ${\cal Q}$. Hence, by Bendixson's theorem \cite{khalil},
there are no closed orbits lying entirely in the simply connected
set ${\cal Q}$, leading to a contradiction.

\end{proof}

We next consider an $N$-oscillator network (for any finite $N$) with arbitrary interconnection topology, non-uniform coupling strengths and arbitrary, but constant natural frequencies. We first identify a positively invariant set that contains the equilibrium set of interest (stable equilibria), and proceed to derive the non-tangency based semistability result.

\section{N-Oscillator network}
\label{Nmodel} A system of $N$ coupled Kuramoto oscillators with
 is considered here
\beq
    \dot{\theta}_i = \omega_i + \sum_{j=1, j\neq i}^{N} \frac{K_{ij}}{N} \sin(\theta_j - \theta_i )
\label{osc_network}
\eeq
where $\theta_i\in S^1$ is the phase of the $i$-th oscillator, $\omega_i$ is its natural frequency
and $ K_{ij}\in \R$ is the coupling gain. We assume symmetric coupling $(K_{ij}=K_{ji})$ between a pair of
oscillators and $\omega_i$ is a constant. Define $\Upsilon\deff \{(K_{11},K_{12},\ldots,K_{(N-1)N}): K_{ij}\ge 0, i<j\}$.  Let $K = \frac{1}{N} \diag(\tilde{K})$, the coupling strength
matrix with $\tilde{K} =[K_{12}, K_{13},\ldots ,K_{(N-1)N}]^\top$.  The set $\Upsilon$ represents all possible  interconnection topologies.

 If the oscillators are considered as nodes with every node
connected to every other node, then the nodes form a graph $ G$ with $N$ vertices and $e = {{N}\choose{2}}$ edges. The incidence matrix $B\in \R^{N\times e}$ of an oriented graph ${G}$ is constructed such that $B_{ij}=-1$ if the edge is incoming to the vertex $i$, $B_{ij}=1$ if the edge is outgoing  to the vertex $i$, and $0$ otherwise. We consider the  incidence matrix corresponding to an all-to-all connectivity and generate every possible interconnection topology by an appropriate choice of $\tilde{K}$. In the rest of the paper, $B$ has the fixed form
\beqn B=
\left[
\begin{array}{cccc}
1&1&\cdots &0 \\
-1&0&\cdots &0 \\
0&-1&\cdots &0 \\
\vdots & \vdots & \ddots & \vdots\\
0&0&\cdots &1\\
0&0&\cdots &-1
\end{array}\right].
\eeqn
The objective of our analysis is to obtain conditions on the interconnection topology such that  the oscillator network \eqref{osc_network} synchronizes.

Rewriting \eqref{osc_network} in vector form

\beq
    \dot{\theta} = \omega - B\; K \sin(B^\top \theta)
\label{eqn:system}
\eeq
where $\theta = [\theta_1, \theta_2, \ldots, \theta_N]^\top \in
\Gamma^N$,
With $X = (x_1,x_2,\ldots, x_e)^\top \deff \Delta \theta = B^\top
\theta \in \Gamma^e, V\deff  B^\top \dot{\theta}$, \eqref{eqn:system} takes the form

\beq
\left(
\begin{array}{l}
    \dot{X} \\
    \dot{V}
\end{array}\right)=F(X,V)=\left(
\begin{array}{l}
     V \\
   G(X)V
\end{array}\right)
\label{N_system} \eeq
 where $G(X) =
-B^\top B K\; \diag(\cos(X))$.
\brm Since $\rank(B^\top)=N-1$, the states $x_i,i=N,\ldots,e$ are
linear combinations of $x_i,i=1,\ldots, N-1$.
 \erm

\brm
Since $V=B^\top \dot{\theta}$ and $\Null (B^\top)=\mathrm{span} \{{\pmb 1}_N\}$, we note that $V=0$ in \eqref{N_system} corresponds to the synchronized condition in \eqref{eqn:system}.
\erm

 The following result will be useful
in the ensuing section.

 \bl \label{VGV}
 For all $X\in (-\pi/2,\pi/2)^e \cap \Col(B^\top), V\in
\Col(B^\top) $, $ V^\top G(X) V\leq 0$. \el
\begin{proof}
 \beqn
\begin{array}{lcl}
   V^\top G(X)V &=& -V^\top B^\top B K \diag(\cos(X)) V \\
    & =& - \dot{\theta}^\top B B^\top B K \diag(\cos(X))B^\top \dot{\theta} \\
    & =& - \dot{\theta}^\top BB^\top BP(X)B^\top \dot{\theta}
    \end{array}
\eeqn
Note that $P(X) \deff K\diag(\cos(X))$ is a positive semi-definite
diagonal matrix. Since $BB^\top B=NB$, we have
 $- \dot{\theta}^\top BB^\top BP(X)B^\top \dot{\theta}=- N\dot{\theta}^\top BP(X)B^\top \dot{\theta}\leq 0$.
\end{proof}
 The set of equilibria  of \eqref{N_system} is ${\mathcal
E} = \{(X,V):  V= 0 \}$. Linearizing
\eqref{N_system} about an equilibrium point $(X^{\ast},0)$,

\beq
    \begin{pmatrix} \dot{\tilde{X}} \\ \dot{\tilde{V}} \end{pmatrix} = \underbrace{\begin{bmatrix} 0_{e \times e} & I_{e
    \times e} \\ 0_{e \times e} & G(X^{\ast}) \end{bmatrix}}_{A} \begin{pmatrix} \tilde{X} \\ \tilde{V} \end{pmatrix}
\label{eqn:lin_dyn}
\eeq
where, $\tilde{X} = X - X^{\ast}$ and $\tilde{V} = V$. Further, for
all $X\in(-\pi/2,\pi/2)^e $, all the non-zero eigenvalues of $A$ are
negative. Thus, the equilibrium set of interest for semistability analysis is ${\mathcal
E}_s =(-\pi/2,\pi/2)^e\subset {\cal E}$.   Note that $X \in \Col(B^\top) \subset
R^e, V \in \Col(B^\top) \subset R^e$ and  $\rank(B^\top)
={N-1}$.

To proceed with the analysis, we define  a set ${\cal H}$ that is
positively invariant through the following Lemma. Let $e_i\in \R^e$
be the $i$th basis vector from the canonical basis.
\bl
The set
 \beqn
{\cal H} = \lbrace (X,V) : X \in \left( -\frac{\pi}{2},
\frac{\pi}{2} \right)^e \cap \Col(B^\top), V \in {\cal J} \rbrace \eeqn is
positively invariant along the trajectories of \eqref{N_system},
where
\beqn
\begin{array}{lll}
{\cal  J} = \lbrace V \in R^e: V = (B^\top \omega - B^\top B K
\sin(X));\\ \hspace{1cm} X \in \left( -\frac{\pi}{2}, \frac{\pi}{2} \right)^e
\cap \Col(B^\top); \\ \hspace{1cm} |e_i^\top B^\top \omega|\leq \frac{2}{N} \tilde{K}_i + \frac{1}{N} \sum_{\stackrel{j=1}{j \neq i}}^e |(B^\top B)_{ij}| \tilde{K}_{j} \sin\vert x_j\vert,\\
\hspace{1cm} i=1,\ldots,e \rbrace.
\end{array}
\eeqn
\label{pi_H}
\el
\begin{proof}
Note  that  there are $2e$ bounding surfaces for $X \in \left(
-\frac{\pi}{2}, \frac{\pi}{2} \right)^e$. The $i$th bounding surface
is characterized by $x_i = \frac{\pi}{2}$, $i = 1,\ldots,e$. The
boundary of ${\cal H}$ is given by $\partial{\cal H}= B_{X}\cup
B_V$, where $B_X=\cup_{i=1}^{2e} B_{X_i}$ with
\beqn B_{X_i}=\left\{
\begin{array}{lcl}
\left\{(X,V):x_i=\frac{\pi}{2}\right\}&,&i=1,\ldots,e\\
 \left\{(X,V):x_i=-\frac{\pi}{2}\right\}&,&i=e+1,\ldots,2e.
 \end{array}\right.
\eeqn
and $B_V$ to be defined later in the proof.
The outward normal to the bounding surface characterized by $x_i =
\frac{\pi}{2}$ is $e_i$. For  ${\mathcal H}$ to be positively
invariant,
\beq \left[\begin{matrix} e_i^\top & 0\end{matrix}\right]
\left[\begin{matrix} V \\ G(X)V \end{matrix}\right] \leq 0
\label{cond1} \eeq
 on the bounding surface.
Inequality \eqref{cond1} yields $e_i^\top V \leq 0$. For $V = B^\top
\omega - B^\top B K \sin(X)$, we get
\beq
    e_i^\top B^\top\omega \leq  e_i^\top B^\top B K \sin(X)
\label{cond2} \eeq where, $x_i = \frac{\pi}{2}$, and imposing the
condition $x_j\in [-\pi/2,\pi/2]$, we observe that
\beqn x_j \in\left\{
\begin{array}{lclll}
\left[0, \frac{\pi}{2} \right]\; &\mbox{if}&\; (B^\top B)_{ij} &= &1\\
\\
\left[-\frac{\pi}{2},0 \right]& \mbox{if}& (B^\top B)_{ij}& =& -1\\\\
\left[-\frac{\pi}{2},\frac{\pi}{2} \right]&\mbox{if}&(B^\top B)_{ij}
&=& 0.
\end{array}
\right. \eeqn
Similarly, the outward normal to the bounding surface characterized
by $x_i = -\frac{\pi}{2}$ is $-e_i$. Inequality \eqref{cond1} is
\beq
    -e_i^\top (B^\top \omega - B^\top B K\sin(X)) \leq 0
    \label{cond7}
\eeq for $x_i = -\frac{\pi}{2}$, and
\beqn x_j \in\left\{
\begin{array}{lclll}
\left[0, \frac{\pi}{2} \right]\; &\mbox{if}&\; (B^\top B)_{ij}& =& -1\\
\\
\left[-\frac{\pi}{2},0 \right]& \mbox{if}& (B^\top B)_{ij}& =& 1\\\\
\left[-\frac{\pi}{2},\frac{\pi}{2} \right]&\mbox{if}&(B^\top B)_{ij}
&=& 0.
\end{array}
\right. \eeqn
 Condition \eqref{cond7} can be rewritten as
\beq
 e_i^\top B^\top \omega \geq -e_i^\top B^\top B K \sin(X) \label{cond3} \eeq
where, $x_i = \frac{\pi}{2}$, and \beqn x_j \in\left\{
\begin{array}{lclll}
\left[0, \frac{\pi}{2} \right]\; &\mbox{if}&\; (B^\top B)_{ij}& =& 1\\
\\
\left[-\frac{\pi}{2},0 \right]& \mbox{if}& (B^\top B)_{ij} &= &-1\\\\
\left[-\frac{\pi}{2},\frac{\pi}{2} \right]&\mbox{if}&(B^\top B)_{ij}
&=& 0.
\end{array}
\right. \eeqn

 From \eqref{cond2} and \eqref{cond3}, it follows that
 \beq
    |e_i^\top B^\top \omega| \leq e_i^\top B^\top B K \sin(X)
    \label{cond4}
\eeq
where, $x_i = \frac{\pi}{2}$, and \beqn x_j \in\left\{
\begin{array}{lclll}
\left[0, \frac{\pi}{2} \right]\; &\mbox{if}&\; (B^\top B)_{ij} &=& 1\\
\\
\left[-\frac{\pi}{2},0 \right]& \mbox{if}& (B^\top B)_{ij}& =& -1\\\\
\left[-\frac{\pi}{2},\frac{\pi}{2} \right]&\mbox{if}&(B^\top B)_{ij}
&=& 0.
\end{array}
\right. \eeqn
Equation \eqref{cond4} can be expanded as follows
\beq
\begin{array}{lcl}
    |e_i^\top B^\top \omega|&\leq & e_i^\top B^\top B K \sin(X)\\
    & =& \frac{2}{N} \tilde{K}_i + \frac{1}{N} \sum_{\stackrel{j=1}{j \neq i}}^e |(B^\top B)_{ij}| \tilde{K}_{j} \sin\vert x_j\vert
\end{array}
    \label{cond5}
\eeq
where,  $x_j \in \left[-\frac{\pi}{2}, \frac{\pi}{2}\right]$. From \eqref{cond5},
a conservative bound (for which the set ${\cal H}$ is positively
invariant) can be obtained as
\beqn
    |e_i^\top B^\top \omega| \leq \frac{2}{N} \tilde{K}_i< \min(e_i^\top B^\top B K \sin(X)).
\eeqn
%
This yields the following sufficient condition on the  coupling
strength
\beqn
    \tilde{K}_i \geq \frac{N}{2}|e_iB^\top \omega|
\eeqn
 where $\tilde{K}_i$ is the coupling strength corresponding to the $i$th edge and  $|e_iB^\top \omega|$ is
 the magnitude of the difference in natural frequencies between the oscillators at the vertices corresponding to the $i$th
 edge.\\
The boundary $B_V$ is a level set of the form $\eta(X,V) =
\frac{1}{2} (V + B^\top B K \sin(X))^\top (V + B^\top B K \sin(X)) =
C_1$, where $C_1=\omega_0^\top B B^\top \omega_0$, and  $\omega_0$
satisfies $|e_i^\top B^\top \omega_0|=\frac{2K_i}{N}, i=1,\ldots,e$.
The normal to the level set of  $\eta(X,V)$ is
\beq \nabla_{(X,V)} \eta = (V + B^\top B K \sin(X))^\top \left[
\begin{matrix} B^\top B P(X) & I_{e \times e} \end{matrix}
\right] \label{normal} .\eeq
 The dot product of the vector field  $F$ in \eqref{N_system} and the normal \eqref{normal},
 at any point on
the boundary $B_V$ is given by
\beqn
     (V + B^\top B K
    \sin(X))^\top \left[ \begin{matrix} B^\top B P(X) & I_{e
    \times e} \end{matrix} \right]\left[ \begin{matrix} V \\ G(X)V
    \end{matrix} \right]\\
     = (V + B^\top B K \sin(X))^\top(-G(X)V + G(X)V) = 0.
\eeqn
 Therefore, for every point on the boundary $B_V$, the vector field $F$ is tangential  to it. Therefore, the set
 ${\cal H}$ is positively invariant.
\end{proof}

The bound on the critical values of the coupling strengths below which the
trajectories of \eqref{N_system} will not be bounded by ${\cal H}$
for any initial condition, are given by
\begin{align}
    |e_i^\top B^\top \omega| & \leq \max(e_i^\top B^\top B K\sin(X)) \notag \\
     & = \max(\frac{2}{N} \tilde{K}_i + \frac{1}{N} \sum_{j \neq i} (B^\top B)_{ij}\tilde{K}_{j}\sin(x_j)) \notag \\
    & < \frac{2}{N} \tilde{K}_i + \frac{1}{N} \sum_{j \neq i} |(B^\top B)_{ij}|\tilde{K}_{j}.
    \label{cond6}
\end{align}
For networks with uniform coupling ($\tilde{K}_i = K_{0}$),
\eqref{cond6} reduces to
\beq
\begin{array}{lcl}
\max(|e_i^\top B^\top \omega|) & \leq & \frac{2K_0}{N} + \frac{K_0}{N} \sum_{j \neq i} |(B^\top B)_{ij}| \\
 & = & \frac{K_0}{N} \left( 2+\left({N\choose 2}-1-{N-2\choose 2}\right)\right) \\
 & = & \frac{2 K_0 (N-1)}{N}
 \end{array}
 \eeq
which can re-expressed as
\beq K_0 \geq \frac{N}{2(N-1)} \max(| e_i^\top B^\top \omega |) =
\frac{N}{2(N-1)} \| B^\top \omega \|_{\infty}. \label{bound} \eeq
The bound on the critical coupling gain $K_0$, in  \eqref{bound} is same as that derived by Jadbabaie et. al in \cite{Jadbabaie} for the onset of synchronization in uniform networks with all-to-all coupling. The
bounds for critical coupling derived in this paper are a
generalization of this result for networks with non-uniform
coupling and arbitrary interconnection topology.
%
%
 In summary, we have derived the necessary and sufficient conditions on the coupling strengths for the positive invariance of set $\cal H$ (the set of phase angles for which $(\theta_i - \theta_j) \in [-\pi/2,\pi/2]$). The following subsection utilizes this result to prove the semistability of \eqref{N_system}, therefore synchronization in the Kuramoto model \eqref{osc_network}.

\subsection{Nontangency and semistability}
For every $X \in \mathcal{E}_s$, the tangent cone  is
$T_X\mathcal{E}_s =\{(\gamma d,0):d \in {\cal E}_s,\gamma\in \R \}$.
 For an equilibrium point $(X,0) \in {\cal E}_s $, let $D$ be an
 open and bounded neighbourhood of $X$ defined as
\beqn \begin{array}{lcl}
 D&=&\{(x,v)\in \R^e\times \R^e
:\vectornorm{(X,0)-(x,v)} <
 \epsilon(X),\\
 &&\epsilon(X)>0\}\subset {\mathcal H}. \end{array}\eeqn
  Then, an outer estimate of the  direction cone  is
given by
\beqn  \hat{\mathcal{F}}_X= \left\lbrace \lambda(v,G(x)v): (x,v) \in
D, \lambda>0 \right\rbrace. \eeqn
We next show that the nontangency between the vector field $F$  and the set of equilibria ${\mathcal E}_s$ holds
through the following Lemma.
\begin{lemma}\label{nontangent}
For every $X\in {\mathcal E}_s$, $T_{X}\mathcal{E}_s \cap
\hat{\mathcal{F}}_X = \{0\}$ if the network graph corresponding to \eqref{osc_network} is connected. \label{nontangency}
\end{lemma}
\begin{proof} From the second line of \eqref{N_system}, \
\beqn \dot{V} = G(X)V = -B^\top B K \diag(\cos(X)) B^\top
\dot{\theta}.\eeqn
We  claim  that if the network graph corresponding to \eqref{osc_network} is connected, $\dot{V} = 0$ if and only if $V = 0$. If $V =
0\implies \dot{V}=G(X)V=0$. Conversely,
\beqn
\begin{array}{lll}
 \dot{V}&=&-B^\top B PB^\top \dot{\theta}=0\\
&\implies & \dot{\theta} \in \Null{(B^\top B PB^\top)}\\
 &\implies & \dot{\theta} \in \Null{(B^\top B PB^\top)}=\Null{( B PB^\top)}.
\end{array}
\eeqn
Since $P(X)$ is a positive semi-definite diagonal matrix, we let $P = P^{\frac{1}{2}}P^{{\frac{1}{2}}^\top}$, where $P^{\frac{1}{2}}_{ii} = \sqrt{P_{ii}}$ and $P^{\frac{1}{2}} = P^{{\frac{1}{2}}^\top}$. Hence, $\Null(BPB^\top) = \Null((BP^{\frac{1}{2}})(BP^{\frac{1}{2}})^\top) = \Null(P^{\frac{1}{2}}B^\top)$. Therefore,
\begin{align*}
	\dot{\theta} \in \Null(B^\top B P B^\top) \Rightarrow \dot{\theta} \in \Null(P^{\frac{1}{2}} B^\top)
\end{align*}
\noindent $	\Rightarrow \dot{\theta} \in \Null(B^\top)$ or $B^\top \dot{\theta} \in \Null(P^{\frac{1}{2}})=\Null(P)$. Thus, $B^\top \dot{\theta} \in \Null(P)$ only if $\Col(B^\top) \cap \Null(P) \neq \emptyset$. Further,
\begin{align*}
	\Null(P) = \Span\{ e_i \in \R^e: \tilde{K}_i = 0 \} \\
	\Col(B^\top) = \Span \{ B_i \in \R^e : 1 \leq i \leq N-1 \}
\end{align*} where $B_i$ is the $i$-th row vector of $B$. We note that for connected graphs, $\dot{V}=0$ if and  only if $V=0$ ( $\Col(B^\top) \cap \Null(P) =\emptyset$ if and only if the graph is connected).

The intersection of  $T_{X}\mathcal{E}_s$ with $
\hat{\mathcal{F}}_X$ yields
\beqn
\begin{array}{rll}
    \lambda v &=& \gamma d \\
    \lambda G(x)v &= &0
    \end{array}
\eeqn
Since  $G(x)v = 0$ if and only if  $v=0$ it follows that
${T_{X}\mathcal{E}_s \cap \hat{\mathcal{F}}_X = \{0\}}$. \end{proof}

The following result establishes the semistability of
\eqref{N_system}.
\bp Every equilibrium in $\mathcal{E}_s$ of \eqref{N_system} is
semistable relative to $\mathcal{H}$.
\label{semistability_N}
\ep
\begin{proof}
Consider the continuously differentiable function $V_2:\mathcal{H}
\longrightarrow \R$ defined by $V_2(x)=\frac{V^\top V}{2}$.
The derivative of $V_2$ along the trajectories of \eqref{N_system}
is $\dot{V}_2 =  V^\top G(X) V\le 0$   for every $(X,V) \in
\mathcal{H}$, which follows from Lemma \ref{VGV}.
The set of points where the derivative of $V_2$ is zero is given by
$\dot{V}_2^{-1}(0) = \{(X, V) \in \mathcal{H} : V = 0 \}$. This claim follows from Lemma \ref{VGV} and  the arguments in Lemma \ref{nontangent}. The
largest negatively invariant subset of
$\overline{\dot{V}_2^{-1}(0)}$ is $ \mathcal{E}_s$. Let $\mathcal{L}
\triangleq \mathcal{H} \setminus \overline{\dot{V}_2^{-1}(0)}$.
Clearly, every equilibrium $Z \in \mathcal{E}_s$ is a local
maximizer of $\dot{V}_2$ relative to $\mathcal{H}$ and a local
minimizer of $V_2$ relative to $\mathcal{L}$.

Consider an equilibrium $X \in \mathcal{E}_s$. There exists a
relatively open neighbourhood $\mathcal{U} \subseteq \mathcal{H}$ of
$X$ such that $\mathcal{U} \cap \mathcal{E} = \mathcal{U} \cap
\mathcal{E}_s$. It now follows that every $Z \in \mathcal{U} \cap
\mathcal{E}$ is a local maximizer of $\dot{V}_2$ relative to
$\mathcal{H}$ and a local minimizer of $V_2$ relative to
$\mathcal{L}$. Moreover, from Lemma \ref{nontangency}, the vector
field $F$ in \eqref{N_system} is nontangent to $\mathcal{U} \cap
\mathcal{E}$ at every point $Z \in \mathcal{V} \cap \mathcal{E}$
relative to $\mathcal{H}$.
Now, by applying (iii) of Corollary \ref{cor:7.2}, every $X \in
\mathcal{E}_s$ is semistable relative to $\mathcal{H}$.
\end{proof}
We next show that ${\cal H}$ is an attracting set in ${\cal M}$, where
 \begin{align*}
{\cal M} \deff \left\{ (X,V) : X \in  \left( -\pi+\delta, \pi-\delta \right]^e  \cap \Col(B^\top), \delta>0, V \in {\cal J} \right\}. \end{align*}
\bp
\label{attractor}
There exists a non-empty, connected, compact and positively invariant set ${\cal N}\subseteq {\cal M}$ containing ${\cal H}$ such that ${\cal H}$ is an attracting set of \eqref{N_system},
if the coupling gains satisfy the condition
\beq
\begin{array}{lll}
&& \hspace{-1cm} \frac{1}{N} \sum_{i=1, \tilde{K}_i > 0}^{e}  (2 \tilde{K}_i - (N-2)\Delta_m ) \sin \vert \delta \vert   \\
   &>& \sum_{i=1}^{e} \vert e_i^{\top} B^\top \omega \vert  + \frac{1}{N} \sum_{i=1, \tilde{K}_i = 0}^{e}  (N-2)\Delta_m
\end{array}
\label{attracting_condition}
\eeq
where the non-zero $\tilde{K}_i$s satisfy $\tilde{K}_i \ge \frac{(N-2)}{2} \Delta_m$, and $\Delta_m=({\tilde{K}_{\max}}-{\tilde{K}_{\min}})$.
\ep
\begin{proof}
Consider a locally Lipschitz and regular potential function of the form
\beqn
	V_3(X) = \sum_{i=1}^{e} \vert x_i\vert  - (N-1)\frac{\pi}{2}
\eeqn
defined over ${\cal M}$.
Note that $V_3(X) > 0,\;\;\forall \;X\in {\cal M}\setminus {\cal \bar{H}}$ and $V_3(X) =0$ on a compact set contained in ${\cal \bar{H}}$. The generalized gradient  of $V_3$ is
\beqn
	\frac{\partial V_3}{\partial x_i} = \left\{
\begin{array}{lcl}
\sign(x_i) &\mathrm{if} & x_i \neq 0 \\
	  \left[-1,1\right] & \mathrm{if}&  x_i = 0
\end{array}\right.
\eeqn
The set-valued Lie derivative \cite{cortes_discontinuous}, denoted by $\tilde{{\cal L}}_FV_3(X)$ of $V_3$  is obtained as follows.
\beqn
\tilde{{\cal L}}_FV_3(X)=
\left\{
\begin{array}{lll}
 \sum_{i=1}^{e} \sign(x_i)v_i &{ \mbox if} &  x_i \neq 0 \; \forall i \in \{1,\ldots,e \}\\
\emptyset  & {\mbox if } &  v_i\ne 0\; {\mbox when }\; x_i= 0\\
 \sum_{i=1}^{e} \sign(x_i)v_i & {\mbox if } &  v_i=0 \; {\mbox when}\; x_i= 0.
\end{array}
\right.
\label{der_nonsmooth_potential}
\eeqn where $V = (v_1,\ldots,v_e)^\top$.
$\cal{H}$ is an attracting set in $\cal{M}$ if $\tilde{{\cal L}}_FV_3(X) < 0$ $\forall$ $X \in {\cal{M}} \setminus \bar{\cal {H}}$ and $V \ne 0$, which is true if the following condition holds.
\beqn
\sum_{i=1}^{e} (\sign(x_i) e_i^\top B^\top \omega - \frac{1}{N} \left[ \sum_{j=1}^{e} (B^\top B)_{ij} \sign(x_j) \tilde{K}_j \right]\times\\
\sign(x_i) \sin\vert x_i \vert) < 0
\eeqn
where $\sin(\vert X \vert) \deff  (\vert \sin x_1\vert,\ldots,\sin\vert x_e \vert)^\top$.
\beqn
\begin{array}{lll}
 \frac{1}{N} \sum_{i=1}^{e} \sum_{j=1}^{e} \hspace{-1cm}& & (B^\top B)_{ij}\sign(x_j)  \tilde{K}_j \sin x_i )\\ &\ge&  \frac{1}{N} \sum_{i=1}^{e}  (2 \tilde{K}_i - (N-2)\Delta_m ) \sin \vert x_i \vert   \\ & >& \sum_{i=1}^{e} \vert e_i^{\top} B^\top \omega \vert  \\
 &\ge&  \sum_{i=1}^{e} (\sign(x_i) e_i^\top B^\top \omega \end{array}
\eeqn
which yields the sufficient condition \eqref{attracting_condition}. Finally, the set ${\cal N}$ is  characterized as
\begin{align*}{\cal N} \deff \{ (X,V) : X \in  \left( -\pi+\delta, \pi-\delta \right]^e  \cap \Col(B^\top), \\  \sum_{i=1}^{e} \vert x_i\vert  - (N-1)\pi \leq 0, V \in {\cal J} \}. \end{align*}
\end{proof}

%
We end this section with the derivation of  the synchronized frequency through the following Lemma.
\bl
For all initial conditions in $\mathcal{N}$, the angular frequencies of \eqref{eqn:system} synchronize to the mean of the natural frequencies of the oscillators.
\el
\begin{proof}
Through Propositions \ref{semistability_N} and \ref{attractor}, it was established that for all initial conditions in $\mathcal{N}$, the oscillators synchronize, which corresponds to $V=0$ in \eqref{N_system}. This further implies that \beq V = B^\top \omega - B^\top B K \sin(X^*) = 0 \label{sync_condition_1} \eeq where $(X^*,0)$ is the fixed point of  \eqref{N_system}. From \eqref{eqn:system}, we obtain \beq \Omega^*{\pmb 1}_N\deff\displaystyle{\lim}_{t\rightarrow \infty} \dot{\theta}(t) = \omega - BK \sin(X^*). \label{sync_condition_2}. \eeq  Pre-multiplying \eqref{sync_condition_1} by $\frac{1}{N} B$, we obtain $\frac{1}{N}BB^\top K \sin(X^*) = BK \sin(X^*) = \frac{1}{N} BB^\top \omega$. Equation \eqref{sync_condition_2} reduces to $\Omega^*{\pmb 1}_N = (I - \frac{1}{N} BB^\top) \omega = \left( \frac{1}{N} \sum_{i=1}^{N} \omega_i \right) {\pmb 1}_N$.
\end{proof}
To summarize our results in this section, we have established that there always exists a set of coupling strengths for all natural frequencies of the oscillators, such that for all initial conditions $X(0) \in {\cal N}$ the oscillators synchronize, provided the network graph is connected.
\section{Simulation Results}
\label{num} We first consider an open-chain network with $N=3$ and non-uniform coupling $K=\mathrm{diag}\{3,2,0\}$ and  natural frequencies $\omega=(1,2,3)$.
With $X\deff (x_1,x_2,x_3)$, the reduced-order system (note that
$x_3=x_2-x_1$) can rewritten as
 \beqn
 \dot{x}_1&=&-1-6\sin(x_1)-2\sin(x_2)\\
 \dot{x}_2&=&-2-3\sin(x_1)-4\sin(x_2).
 \eeqn
The stream plot for this case is shown in
 Figure \ref{3d1}. In the interval $X\in(-\pi,\pi]^3$, there exists only one fixed point at $(x_1,x_2)=(0,-\pi/6)$ that corresponds to the synchronized state. In this case, we observe that ${\cal N}={\cal M}$.
 \begin{figure}[h]
\centering {\includegraphics[scale=0.6]{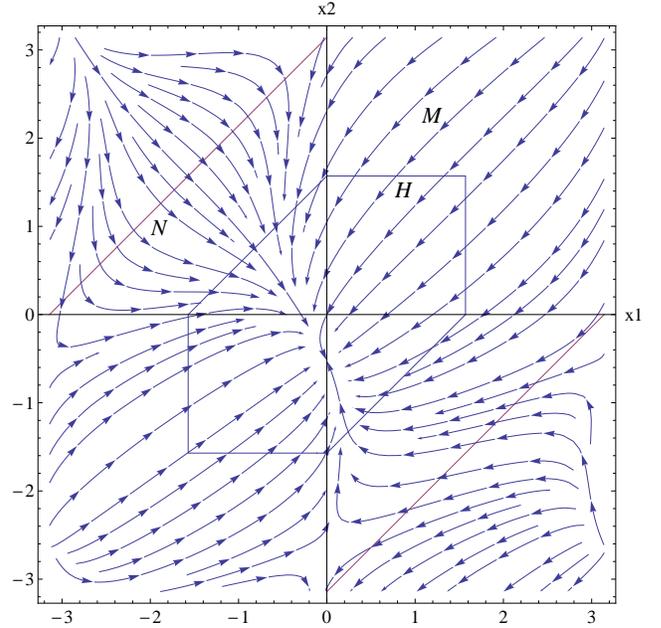}}
\caption{Stream plot of the three-oscillator system } \label{3d1}
\end{figure}
We next consider a network (see Figure \ref{network}) with $N=5$ and non-uniform coupling and   natural frequencies $\omega=(1,2,3,4,5)$.
 \begin{figure}[h]
\centering {\includegraphics[scale=0.6]{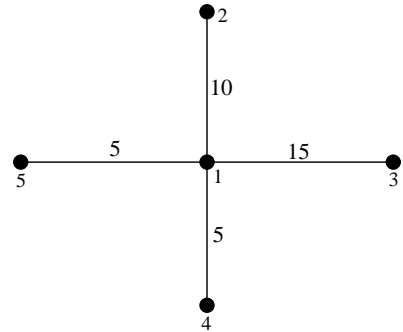}}
\caption{Network graph of five oscillators } \label{network}
\end{figure}
The plot of angular frequencies versus time is shown in Figure \ref{3d2} for the initial condition $\theta(0)=(-2\pi/3,2\pi/3,\pi/3,-\pi/6,0)$, where the frequencies
 synchronize to the mean  $\Omega^*=\unit{3}\radian\per\second$.
\begin{figure}[h]
\centering {\includegraphics[scale=0.5]{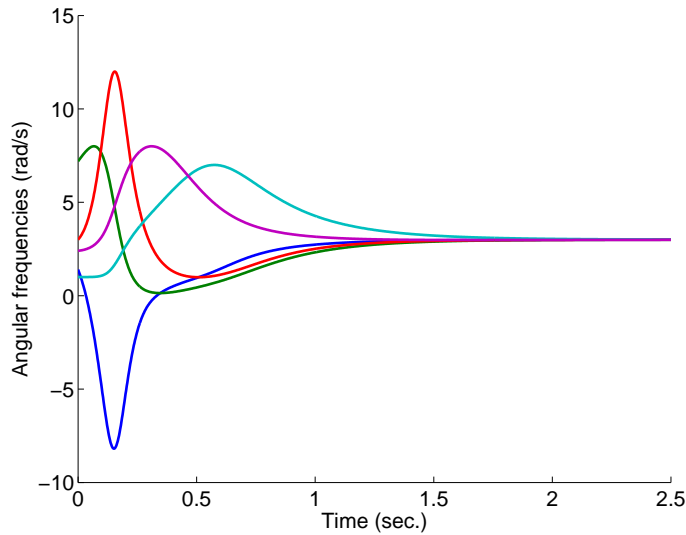}}
\caption{Time-response of angular frequencies } \label{3d2}
\end{figure}
\section{Conclusions}
\label{con} The objective of this work was to obtain convergence
results for the Kuramoto  model, which was presented in the
framework of semistability theory. For illustrative purpose, these
results were obtained for the two-oscillator case. In arriving at the semistability result, the nontangency between the vector field and
the tangent space of the set of equilibria was established by using
a novel method for obtaining an estimate  of the  outer bound of the
direction cone.
In the $N$-oscillator case we consider networks with connected graphs, arbitrary interconnection topology, non-uniform coupling strengths and non-identical natural  frequencies. We establish that such a network synchronizes under certain conditions on the coupling gains which have been  explicitly derived.
\bibliographystyle{ieeetr}
\bibliography{references}

\begin{thebibliography}{10}

\bibitem{Strogatz}
S.~H. Strogatz, ``{From Kuramoto to Crawford: Exploring the onset of
  synchronization in populations of coupled oscillators},'' {\em Physica D},
  vol.~143.

\bibitem{MirolloStrogatz}
R.~E. Mirollo and S.~H. Strogatz, ``{Synchronization of pulse-coupled
  biological oscillators},'' {\em SIAM Journal on Applied Mathematics},
  vol.~50, no.~6, pp.~1645--1662, 1990.

\bibitem{Jadbabaie}
A.~Jadbabaie, N.~Motee, and M.~Barahona, ``{On the stability of the Kuramoto
  model of coupled nonlinear oscillators},'' in {\em Proceedings of the
  American Control Conference}, pp.~4296--4301, 2004.

\bibitem{spong05}
N.~Chopra and M.~W. Spong, ``{On synchronization of Kuramoto oscillators},'' in
  {\em Proceedings of the 44th IEEE Conference on Decision and Control and the
  European Control Conference 2005}, (Seville, Spain), December 12-15, 2005.

\bibitem{Acebron}
J.~A. Acebron, L.~L. Bonilla, C.~J.~P. Vicente, F.~Ritort, and R.~Spigler,
  ``{The Kuramoto model: A simple paradigm for synchronization phenomena},''
  {\em Reviews of Modern Physics}, vol.~77.

\bibitem{Dorfler}
F.~Dorfler and F.~Bullo, ``{On the critical coupling for Kuramoto
  oscillators},'' {\em SIAM J. Applied Dynamical Systems}, vol.~10, no.~3,
  pp.~1070--1099, 2011.

\bibitem{bullo_Dorfler}
F.~D\"orfler and F.~Bullo, ``Synchronization and transient stability in power
  networks and nonuniform kuramoto oscillators,'' {\em SIAM Journal on Control
  and Optimization}, vol.~50, no.~3, pp.~1616--1642, 2012.

\bibitem{Mirollo}
R.~E. Mirollo and S.~H. Strogatz, ``{The spectrum of the locked state for the
  Kuramoto model of coupled oscillators},'' {\em Physica}, vol.~D, no.~205,
  pp.~249--266, 2005.

\bibitem{Antoine}
A.~Franci, A.~Chaillet, and W.~Pasillas-L\'epine, ``{Phase-locking between
  Kuramoto oscillators: Robustness to time-varying natural frequencies},'' in
  {\em Proceedings of the 49th IEEE Conference on Decision and Control},
  (Georgia, USA), pp.~1587--1592, December 15-17, 2010.

\bibitem{Wang}
Y.~Wang and F.~J. {\relax Doyle III}, ``{Exponential synchronization rate of
  Kuramoto oscillators in the presence of a pacemaker},'' {\em IEEE
  Transactions on Automatic Control}, vol.~58, no.~4, pp.~989--994, 2013.

\bibitem{Maistrenko}
Y.~Maistrenko, O.~Popovych, O.Burylko, and P.~Tass, ``{Mechanism of
  Desynchronization in the Finite-Dimensional Kuramoto Model},'' {\em Physical
  Review Letters}, vol.~93, no.~8, 2004.

\bibitem{Franci}
A.~Franci, A.~Chaillet, E.~Panteley, and F.~Lamnabhi-Lagarrigue,
  ``{Desynchronization and inhibition of Kuramoto oscillators by scalar
  mean-field feedback},'' {\em Mathematics of Control, Signals and Systems},
  vol.~24, pp.~169--217, 2012.

\bibitem{Guy_Bart_Stan}
G.-B. Stan and R.~Sepulchre, ``{Analysis of interconnected oscillators by
  dissipativity theory},'' {\em IEEE Transactions on Automatic Control},
  vol.~52, no.~2, pp.~256--270, 2007.

\bibitem{Scardovi}
L.~Scardovi and R.~Sepulchre, ``{Synchronization in networks of identical
  linear systems},'' {\em Automatica}, vol.~45, no.~11, pp.~2557--2562, 2009.

\bibitem{Trentelman}
H.~L. Trentelman, K.~Takaba, and N.~Monshizadeh, ``{Robust synchronization of
  uncertain linear multi-agent systems},'' {\em IEEE Transactions on Automatic
  Control}, vol.~58, pp.~1511--1523, June 2013.

\bibitem{hartman}
P.~Hartman, {\em Ordinary Differential Equations}.
\newblock Philadelphia, PA: Society for Industrial and Applied Mathematics,
  2002.

\bibitem{bhat_nontangency}
S.~P. Bhat and D.~S. Bernstein, ``{Nontangency-based Lyapunov tests for
  convergence and stability in systems having a continuum of equilibria},''
  {\em SIAM Journal on Control and Optimization}, vol.~42, pp.~1745--1775,
  November 2003.

\bibitem{isidori}
A.~Isidori, {\em Nonlinear Control Systems}, vol.~1.
\newblock Springer Verlag, London, 1995.

\bibitem{bhat_nonsmooth}
Q.~Hui, W.~M. Haddad, and S.~P. Bhat, ``Semistability, finite-time stability,
  differential inclusions, and discontinuous dynamical systems having a
  continuum of equilibria,'' {\em IEEE Transactions on Automatic Control},
  vol.~54, no.~10, pp.~2465--2470, 2009.

\bibitem{khalil}
H.~K. Khalil, {\em Nonlinear Systems}.
\newblock Englewood Cliffs, New Jersey: Prentice-Hall, 3~ed., 2002.

\bibitem{cortes_discontinuous}
J.~Cortes, ``{Discontinuous dynamical systems: A tutorial on solutions,
  nonsmooth analysis and stability},'' {\em IEEE Control Systems Magazine},
  vol.~28, no.~3, pp.~36--73, 2008.

\end{thebibliography}

\end{document}